\documentclass[12pt]{amsart}
\usepackage{amsfonts, epic, amssymb,  
bm, mathptmx, amsmath, array, multirow,url}

\newcommand{\calB}{\ensuremath{\mathcal{B}} }
\newcommand{\calC}{\ensuremath{\mathcal{C}} }

\newcommand{\bfs}{\ensuremath{\bm s} }

\newcommand{\frakf}{\ensuremath{\mathfrak{f}} }

\newcommand{\frakh}{\ensuremath{\mathfrak{h}} }
\newcommand{\fraki}{\ensuremath{\mathfrak{i}} }

\newcommand{\frakm}{\ensuremath{\mathfrak{m}} }
\newcommand{\frakn}{\ensuremath{\mathfrak{n}} }

\newcommand{\frakv}{\ensuremath{\mathfrak{v}} }

\newcommand{\frakz}{\ensuremath{\mathfrak{z}} }

\newcommand{\frakso}{\ensuremath{\mathfrak{so}}          }

\newcommand{\boldF}{\ensuremath{\mathbb F}}

\newcommand{\boldR}{\ensuremath{\mathbb R}}
\newcommand{\boldZ}{\ensuremath{\mathbb Z}}

\newcommand{\ad}{\operatorname{ad}}

\newcommand{\Cay}{\operatorname{Cay}}

\newcommand{\End}{\operatorname{End}}

\newcommand{\myspan}{\operatorname{span}}

\newcommand{\trace}{\operatorname{trace}}

\newcommand{\smallfrac}[2]{\textstyle{\frac{#1}{#2}}  }

\newcommand{\la}{\langle}
\newcommand{\ra}{\rangle}

\theoremstyle{plain}

\newtheorem{theorem}{Theorem}[section] 
\newtheorem{prop}[theorem]{Proposition}
\newtheorem{lemma}[theorem]{Lemma}  
\newtheorem{cor}[theorem]{Corollary}

\theoremstyle{definition} 
\newtheorem{defn}[theorem]{Definition}
\newtheorem{example}[theorem]{Example}

\theoremstyle{remark}
\newtheorem{remark}[theorem]{Remark}
\newtheorem{case}[]{Case}

\newtheorem{problem}[theorem]{Problem}

\begin{document}

\title{Uniform Lie Algebras and Uniformly Colored Graphs} 

\author[T. L. Payne]{Tracy L. Payne}
\address{Department of Mathematics,  Idaho State University, 
Pocatello, ID 83209-8085} \email{payntrac@isu.edu}

\author[M. Schroeder]{Matthew Schroeder}
\address{Department of Mathematics,  Idaho State University, 
Pocatello, ID 83209-8085} \email{schrmat3@isu.edu}


\date{\today}

\thanks{The authors thank the referee for his or her careful
  reading and useful remarks.  
Part of the work in this paper was toward the second author's Master's
thesis \cite{schroeder-thesis} at Idaho State University;  the authors
are grateful to the members of his thesis committee, Shannon
Kobs-Nawotniak, Cathy Kriloff and Yu Chen, for their comments and
suggestions.   The first author thanks Dietrich Burde for 
interesting  discussions.}


\begin{abstract} 
Uniform Lie algebras are combinatorially defined two-step nilpotent Lie
algebras which can be used to define Einstein solvmanifolds. 
These Einstein spaces often have nontrivial isotropy groups.  We
derive basic properties of uniform Lie algebras and we classify
uniform Lie algebras with five or fewer generators.  We define a
type of directed colored graph called a uniformly colored graph and
 establish a correspondence between uniform Lie algebras
and uniformly colored graphs.  We present several methods of
constructing infinite families of uniformly colored graphs and
corresponding uniform Lie algebras.
\end{abstract}

\maketitle

\section{Introduction}
\label{intro}

A Riemannian manifold is Einstein if its Ricci form is a scalar
multiple of its metric.  One approach to the study of Einstein
manifolds is to focus on classes of Einstein manifolds with symmetry,
such as homogeneous spaces and spaces of cohomogeneity one (see
\cite{wang-symmetry-1, wang-symmetry-2}).  We are interested in a
special class of Einstein homogeneous spaces: Einstein solvmanifolds.
A solvmanifold is a simply connected solvable group endowed with a
left-invariant Riemannian metric.   

 In \cite{eberlein-heber}, a representation of a compact Lie group is
used to define a two-step metric nilpotent Lie algebra, which can then be used
to define an Einstein solvmanifold. In \cite{eberlein-heber,
eberlein-representation, leukert-98}, and \cite{lauret-99}, the
geometry of these solvmanifolds was studied.  Such nilpotent Lie
algebras have large automorphism groups and the affiliated Einstein
solvmanifolds have isotropy groups with positive dimension.  In his
earlier study of Einstein solvmanifolds \cite{deloff79}, DeLoff
defined a class of nilpotent Lie algebras called {\em uniform Lie
algebras} enjoying combinatorial symmetry instead of algebraic
symmetry (see Definition \ref{def of uniform}). Every uniform Lie
algebra can be used to define an Einstein solvmanifold, and due to the
combinatorial regularity in the definition of uniform Lie algebra, the
Einstein solvmanifolds often have nontrivial isotropy groups which may
now be finite or infinite.  The category of uniform Lie algebras includes
many of the Lie algebras defined by representations of compact Lie
groups, and well-known classes of two-step nilpotent Lie algebras,
such as Heisenberg Lie algebras, two-step free nilpotent Lie algebras,
Iwasawa type nilpotent Lie algebras for rank one symmetric spaces of
noncompact type, and Lie algebras of Heisenberg type.

 In this work, we study uniform Lie algebras.  For the sake of
economy, we do not explicitly discuss the Einstein solvmanifolds
that they determine.  Uniform Lie algebras may be defined over any
field, and may be of interest over general fields from a purely
algebraic perspective, but throughout this work, because of our
geometric motivation, we assume that the field of definition is
$\boldR.$

Let $(V,E)$ be a simple digraph with edge coloring $c: E \to S.$  
Let $V = \{v_i\}_{i=1}^q$ and $S = \{ z_k\}_{k=1}^p,$ and let
$\frakv$ and $\frakz$ be the vector spaces of $\boldR$-linear
combinations of $V$ and $S,$ respectively.  
Setting $[v_i,v_j] = \sum_{k=1}^p \alpha_{ij}^k z_k,$
where
\[  \alpha_{ij}^k = 
\begin{cases}  
1 &  \text{$e_{ij} = (v_i,v_j) \in E$ and $c(e_{ij}) = z_k$ }\\
-1 & \text{$e_{ji} = (v_j,v_i) \in E$ and $c(e_{ji}) = z_k$ }\\
0 & \text{otherwise}\\
\end{cases} ,\]
defines a    two-step nilpotent
Lie algebra structure on $\frakv \oplus \frakz.$

We define a class of edge colored digraphs, called uniformly colored
graphs, and show  that any uniformly colored graph defines a uniform Lie
algebra, and any uniform Lie algebra may be encoded as a uniformly
colored graph.  Uniform  colorings of graphs are (up to a choice
of orientation) equivalent to $H$-decompositions of regular graphs
with $H =K_2 + \cdots + K_2;$ that is, decompositions of regular
graphs into subgraphs all isomorphic to the same disjoint sum $K_2 +
\cdots + K_2,$ where $K_2$ is the complete graph on two vertices.

Algebraic objects are often used to analyze combinatorial and 
topological problems, 
such as in the case of Orlik-Solomon algebras or free
partially commutative monoids, and
conversely, graphs or simplices have been used to define algebraic
objects, as with rooted tree algebras and in Stanley-Reisner theory.
Closer to the topic at hand, Dani and Mainkar defined a class of two-step
nilpotent Lie algebras associated to graphs that they called nilpotent
Lie algebras of graph type (\cite{dani-mainkar-05}), and Mainkar
showed that two nilpotent Lie algebras of graph type are isomorphic if
and only if the graphs they arise from are equivalent
(\cite{mainkar-15}).

  Einstein solvmanifolds defined by
solvable extensions of nilpotent Lie algebras of graph type have been
studied in \cite{lauretwill07, lafuente-13, oscari-14}, and
 the geometry of metric nilpotent Lie algebras defined by
Schreier graphs was addressed
 in \cite{ray-16}.  (Although both of these classes of
nilpotent Lie algebras, graph type and Schreier type, overlap with the
class of uniform Lie algebras studied here, there are no containment
relations between the class of uniform Lie algebras and either of
these classes. See Remark \ref{Dani-Mainkar}.)
  Pseudo H-type algebras were analyzed using
combinatorial and orthogonal designs in \cite{furutani-et-al}.  In
\cite{payne-agag}, it was shown that if $rp-q-p+1>0,$ an
Einstein solvmanifold defined by a  uniform metric
Lie algebra of type $(p,q,r)$ admits nonisometric Einstein
deformations.

We derive basic properties of uniform Lie algebras and uniformly
colored graphs.  We show how algebraic properties of a uniform Lie
algebra translate to graph theoretic properties of the corresponding
uniformly colored graph.  We show in Proposition \ref{totally
geodesic} how totally geodesic subalgebras of a uniform Lie algebra
may be found using the corresponding uniform graph.
  In Proposition \ref{union equals}, we show that graph
unions correspond to concatenations of Lie algebras, and in
Propositions \ref{graph union disjoint color set} and \ref{graph union
same color set} we determine when graph unions of uniform graphs are
again uniform.

We give many examples of uniform Lie algebras and uniformly colored
graphs, some in infinite families.  We show uniform Lie algebras can
be found from well-known combinatorially defined graphs, such as
Kneser graphs, and decompositions of familiar regular graphs, such as
 one-factorizations and near-one-factorizations of
complete graphs. We present a general method of constructing
uniform Lie algebras from Cayley graphs.  These examples indirectly
give a wealth of  new examples of Einstein solvmanifolds. 

 Many of the
examples we give come from graphs with symmetries, and the
corresponding Einstein solvmanifolds inherit those symmetries.  We
show in Proposition \ref{symmetry} that the symmetry group of a
uniformly colored digraph embeds in the automorphism group of the
corresponding uniform Lie algebra.  Thus, the automorphism group of
a uniform Lie algebra may contain a nontrivial finite subgroup, and
the corresponding Einstein solvmanifold  will have that finite group as
a subgroup of its isotropy group.
In addition, the corresponding simply connected Lie groups may have
infranilmanifold quotients.  (For all uniform Lie algebras, the
corresponding simply connected nilpotent Lie group admits a
lattice).

In Theorem \ref{classification thm}  we classify all uniform Lie
algebras with five or fewer generators; a complete list appears in
Table \ref{list of algebras}.   There are only 12 Lie
algebras in the list, demonstrating that uniform Lie
algebras are not so common in low dimensions.   However,
there are many in higher dimensions (See Remark \ref{lots}).

In addition to the examples that we present here,
 many more uniform edge colorings can
be found on well-known graphs and families of graphs.  We
leave it to the interested reader to
construct uniform edge colorings on the Heawood graph and
the Desargues graph; one-skeletons of symmetric polyhedra;
some (but not all) circulant graphs, some (but not all)
bipartite graphs, and graphs arising from incidence
geometries.  All the resulting uniformly colored graphs
define Einstein solvmanifolds, with the symmetry groups of
the colored graphs embedding into the isometry groups of
the manifolds.

We pose two problems.  The same strategy we used for the proof of
Theorem \ref{classification thm} can be used to classify uniform Lie
algebras of type $(p,q,r)$ with $q \ge 6.$ To assist in this goal, and
of independent interest, one could analyze infinite classes of uniform
Lie algebras of type $(p,q,r)$ which exist for all possible $q \ge 6,$
or in the cubic case, all even $q \ge 6.$
\begin{problem}  What are the structure and algebraic properties of
uniform Lie algebras that have an underlying graph which is a cycle?  a bipartite graph?  a complete graph?  a cubic graph?
\end{problem}

This involves a simply stated problem in graph theory.
\begin{problem}  What are the structure and properties of $K_2 +
  \cdots + K_2$-decompositions of regular graphs?
\end{problem}

The rest of the paper is organized as follows.  In Section
\ref{definition-examples}, we give the formal definition of
a uniform (metric) Lie algebra and present some
illuminating examples. In Section \ref{properties}, we
derive basic properties of uniform Lie algebras.  Then, in
Section \ref{graphs}, we define uniformly colored graphs,
give the correspondences between uniform Lie algebras and
uniformly colored graphs, and translate between algebraic
properties of Lie algebras and properties of uniformly
colored graphs.  In Section \ref{examples} we present a
variety of constructions of uniformly colored graphs. 
  Finally in Section \ref{classification}
we classify uniform type Lie algebras with five or fewer
generators.

\section{Definition and some examples}
\label{definition-examples}

The first thing we do is define Lie algebras of uniform
type and give some examples.   The definition may seem complicated at first,
but once a connection is made with Definition \ref{def of graph coloring},
 it will seem quite simple. 
\begin{defn}\label{def of uniform}
Let $p, q,$ and $r$ be positive integers.  
 A real  nilpotent Lie algebra $\frakn$ is said to
be of {\em uniform of type $(p,q,r)$} if there exists a 
basis $\calB = \{ v_i \}_{i=1}^q \cup \{ z_j\}_{j = 1}^p$ of $\frakn$
and  a positive integer $s,$ called the {\em degree}, such that the
following properties hold.
\begin{enumerate}
\item{For all $i, j =1, \ldots , q$ and all $l, m = 1, \ldots , p,$
\begin{enumerate}
\item{$[v_i, z_j] = 0$ and $[z_l, z_m] = 0,$}
\item{$[v_i,v_j] \in \{ 0, \pm z_1, \ldots, \pm z_p\}.$}
\end{enumerate} }\label{onek}
\item{If $[v_i,v_j]$ is nonzero and $[v_i,v_j]=\pm[v_i,v_k],$ then
$v_j=v_k.$}\label{rdistinct}
\item{For all $l =1, \ldots , p,$ there exist exactly $r$ disjoint
pairs $\{v_i,v_j\}$ with $[v_i,v_j]=z_l.$}\label{r}
\item{For all $j =1, \ldots , q,$ there exist exactly $s$ basis
vectors $v_i$ with $[v_i,v_j] \ne 0.$}\label{s}
\end{enumerate} 
The basis $\calB$ is called a {\em uniform basis} for
$\frakn.$ When $\boldF = \boldR,$ it is natural to
endow $\frakn$ with the inner product $Q$ which makes
$\calB$ orthonormal; then we call $(\frakn, Q)$ a {\em uniform
metric Lie algebra}.
\end{defn}
Note that the parameters $p , q, r$ and $s$ in the definition 
 are dependent: the number of nonzero structure constants 
relative to the basis $\calB$ is $sq=2rp.$    The uniform basis 
has rational structure constants, so by Mal'cev's Criterion, 
the simply connected  Lie groups corresponding to 
uniform Lie algebras  always admit compact quotients.  

The most simple examples of uniform Lie algebras are Heisenberg
algebras and free nilpotent algebras. 
\begin{example}\label{Heisenberg}
Let $\frakh_{2n+1}$ be the Heisenberg Lie algebra with
basis $\calB = (\{ x_i \}_{i=1}^{n} \cup \{ y_i \}_{i=1}^{n} )\cup
\{ z \} $ and Lie bracket defined by $[x_{i}, y_i ] = z$ for $i=1,
\ldots, n.$  With respect to the basis $\calB,$ $\frakh_{2n+1}$  is
a uniform Lie algebra of type$(p,q,r) = (1,2n, n)$ with degree $s = 1.$
\end{example}
Note that the uniform basis for $\frakh_{2n+1}$  is not unique; in
fact, there are uncountably many different uniform bases.  

\begin{example}\label{free}
Let $\frakf_{n,2}$ be the free two-step nilpotent Lie algebra on $n$
generators with basis $\{ x_i \}_{i=1}^{n} \cup \{ x_i \wedge x_j
\}_{1 \le i < j \le n}$ and Lie bracket defined by $[x_{i}, x_j ] =
x_i \wedge x_j$ for $1 \le i < j \le n.$ With respect to this
basis, $\frakf_{n,2}$ is a uniform Lie algebra.  The values of the
associated parameters are $(p,q,r) = ( (\begin{smallmatrix} n \\
2 \end{smallmatrix}),n, 1)$ and $s = n-1.$
\end{example}

Algebras of Heisenberg type are two-step metric nilpotent Lie
algebras which have been studied extensively by geometers
(\cite{btv}).  Crandall and Dodziuk have shown that all algebras of
Heisenberg type are uniform (\cite{crandall-dodziuk}).  Here is an
example of a metric Lie algebra of Heisenberg type.
\begin{example}\label{Damek Ricci}
Let $(\frakn, Q)$ be the nilpotent metric 
Lie algebra with orthonormal basis
$\calB = \{ v_1, v_2, v_3, v_4  \} \cup \{ z_1, z_2\}$ and Lie bracket
\begin{equation}\label{DR} 
 [v_1, v_2]= [v_3,v_4] = z_1, [v_2, v_3]=  [v_1, v_4] = z_2.\end{equation}
This Lie algebra is uniform of type $(2,4,2).$
\end{example}
We may generalize Example \ref{Damek Ricci}
to higher dimensions, although   
the examples in  dimension $2r+2 > 6$
are no longer  Heisenberg type.   
\begin{example}\label{Damek Ricci gen}
  For $r \ge 2,$ define the Lie algebra $\frakn(2r+2)$ of dimension
$2r+2$ to have basis $\calB = \{v_i\}_{i=1}^{2r} \cup \{z_1, z_2\}$
and Lie bracket determined by
\begin{align*}
 [v_{2i-1},v_{2i}] &= z_1 \qquad \text{for $i=1, 2, \ldots, r,$} \\ 
[v_{2i},v_{2i+1}] &= z_2 \qquad \text{for $i=1, 2, \ldots, r-1,$}
  \\ 
[v_1, v_{2r}] &= z_2.
 \end{align*} 
With respect to the basis $\calB,$ $\frakn(2r+2)$ is a uniform Lie
algebra of type $(2,2r,r)$ with $s=2.$

We may also change one sign in the defining bracket relations for
$\frakn(2r+2)$ to define the  uniform Lie algebra $\frakn'(2r+2)$ with
uniform basis
$\calB = \{v_i\}_{i=1}^{2r} \cup \{z_1, z_2\}$ and Lie bracket
\begin{align*}
 [v_{2i-1},v_{2i}] &= z_1 \qquad \text{for $i=1, 2, \ldots, r,$} \\ 
[v_{2i},v_{2i+1}] &= z_2 \qquad \text{for $i=1, 2, \ldots, r-1,$}
  \\ 
[v_{2r}, v_{1}] &=   z_2.
 \end{align*} 
\end{example}
The Heisenberg type Lie algebra in the next example is isomorphic
to the Iwasawa type nilpotent Lie algebra in the Iwasawa decomposition
of the Lie algebra of the isometry group of quaternionic hyperbolic space of dimension
8.
\begin{example}\label{quaternionic}
Let $(\frakn, Q)$ be the metric nilpotent Lie algebra with orthonormal basis
$\calB = \{ v_i \}_{i=1}^4 \cup \{ z_j\}_{j=1}^3$ and Lie bracket
\begin{equation}
 [v_1, v_2]= [v_3, v_4] = z_1, [v_1, v_3]=- [v_2, v_4] = z_2,
 [v_1,v_4]=  [v_2, v_3] = z_3.\end{equation}
With respect to $\calB,$ $\frakn$
 is a uniform Lie algebra of type $(3,4,2).$
\end{example}
As already seen in Example \ref{Damek Ricci gen}, it is possible that there are
two or more uniform Lie algebras which have the same set of indices of
nonzero structure constants relative to a uniform basis; these may or
may not be isomorphic.  In the next example we present a uniform Lie
algebra which, with respect to the uniform basis, has the same indices
for nonzero structure constants as the Lie algebra in Example
\ref{quaternionic}.
\begin{example}\label{quaternion associate}
Let $(\frakn, Q)$ be the metric nilpotent Lie algebra with orthonormal basis
$\calB = \{ v_i \}_{i=1}^4 \cup \{ z_j\}_{j=1}^3$ and Lie bracket
\begin{equation}\label{4-3}  
 [v_1, v_2]= [v_3, v_4] = z_1, [v_1, v_3]= [v_2, v_4] = z_2, [v_1,
v_4]= -[v_2, v_3] = z_3.\end{equation}
This metric Lie  algebra is not of Heisenberg type; therefore $\frakn$
is not isomorphic to the Lie algebra in Example \ref{quaternionic}.

This presentation of $\frakn$ is not the most efficient one, in the sense
that more nontrivial brackets of basis vectors appear than is
necessary.  With respect to a different basis 
$\calC = \{u_i \}_{i=1}^4 \cup \{ y_j\}_{j=1}^3,$ the Lie algebra is
given by 
\begin{equation}\label{4-3b} 
 [u_1, u_2]= y_1, [u_1, u_3]= [u_2, u_4] =  y_2,
 [u_3, u_4] = y_3.\end{equation}
The bases $\calB$ and $\calC$ are related by 
\begin{gather*}
 v_1 = u_1 + u_3, v_2 = u_2 + u_4,  v_3 = u_3 - u_1,
 v_4 = u_4 - u_2 \\  
z_1 = y_1  + y_3, z_2 = 2y_2, z_3 = -y_1 + y_3. 
\end{gather*}
Note that the basis  $\calC$ is not orthogonal with respect to $Q.$ 
\end{example}   

 The following  family $\frakm(q)$  of nilpotent Lie algebras arose in the study
of Anosov Lie algebras (Example 4.4,
\cite{payne-anosov}). Because a cyclic group acts
transitively on it, it is
among a class of nilpotent Lie algebras called {\em cyclic Lie
algebras}.
\begin{example}\label{cyclic algebra}
Let $\frakm(q)$ 
be the nilpotent Lie algebra with basis
$\calB = \{ v_i \}_{i=1}^q \cup \{ z_j\}_{j=1}^q$ and Lie bracket
\begin{equation}\label{anosov}
 [v_1, v_2]= z_1, [v_2, v_3]= z_2, \ldots, [v_q, v_1]= z_q.\end{equation} 
With respect to the given basis, the  Lie algebra $\frakm(q)$  is uniform of
type $(q,q,1).$
\end{example}

\section{Properties of uniform Lie algebras}\label{properties}

In this section we derive some basic properties of uniform Lie
algebras.  First we review some standard definitions.  Let
$(\frakn, Q)$ be a two-step nilpotent metric Lie algebra with
center $\frakz,$ and let $\frakv$ be the orthogonal complement to
the center.  For all $z \in \frakz,$ the map $J_{z}$ in $\End
(\frakv)$ is defined by $J_{z} v = \ad_v^\ast z$ for all $v \in
\frakv,$ where $\ad_v^\ast$ is the metric dual of the linear map
$\ad_v.$

If $\frakn$ is uniform with respect to $\calB= \{v_i\}_{i=1}^q \cup
\{z_j \}_{j=1}^p,$ we can use the identity $Q( J_{z_l} v_i, v_j)
= Q([v_i, v_j], z_l) $ to show that
\[  J_{z_l} (v_i) = 
\sum_{j=1}^q \epsilon_{ji}^l v_j, \quad
\text{where} \quad \varepsilon^l_{ji} = \begin{cases} 1 & \text{if
$[v_i, v_j] = z_l$} \\ -1 & \text{if $[v_i, v_j] = -z_l$} \\ 0 &
\text{otherwise} \\
 \end{cases}.  \] 
Since $\frakn$ is uniform with respect to $\calB,$
for fixed $i$ and $l,$ there is at most one $j$ so
that $[v_i, v_j] = \pm z_l.$  Thus we obtain simple formula for
$J_{z_l},$ when $z_l$ is in a uniform basis for a uniform Lie
algebra.  
\begin{prop}\label{JZ} 
 Let $(\frakn, Q)$ be a uniform metric Lie algebra with
uniform basis $\calB = \{ v_i\}_{i=1}^q \cup \{ z_j\}_{j=1}^p.$  Then 
\[ J_{z_l} (v_i) = \begin{cases} 
v_j & \text{if there exists $v_j$ so that $[v_i, v_j] = z_l$} \\ 
-v_j & \text{ if there exists $v_j$ so that  $[v_i, v_j] =- z_l$} \\ 
0 & \text{otherwise} \\ 
 \end{cases}.   \] 
\end{prop}

Fundamental 
properties of uniform Lie algebras are collected in the next theorem.
Recall that the Frobenius inner product on $\End(\boldR^q)$ is defined
by $\la A, B \ra = \trace (AB^T)$ for $A, B \in \End(\boldR^q).$
\begin{theorem}\label{basic properties}
Let $\frakn$ be a uniform Lie algebra of type $(p,q,r)$
with uniform basis $\calB = \{ v_i \}_{i=1}^q \cup \{ z_j\}_{j = 1}^p$
and degree $s.$ 
Let $\frakv = \myspan \{ v_i \}_{i=1}^q$ and let $\frakz = \myspan \{
z_j\}_{j = 1}^p.$  Then $\frakn$ is a two-step nilpotent Lie algebra
with the following properties.
\begin{enumerate}
\item{The center of $\frakn$ is $\frakz,$ 
and $[\frakn, \frakn] = \frakz.$}\label{center}
\item{The centralizer $C(v_i)$ of any vector $v_i$ in $\calB$ is
    spanned by $\{ v_j \, : \, [v_i, v_j] = 0\} \cup
    \{z_j\}_{j=1}^p,$ and it has dimension $p+q - s.$}\label{centralizer}
\item{For all $v_i$ in $\calB,$ the rank of $\ad_{v_i}$ is $s.$
  }\label{s meaning}
\item{For all  $z_j$ in $\calB,$ the rank of $J_{z_j}$ is $2r.$
  }\label{r meaning}
\item{The set of maps $\{ J_{z_j}  \}_{j=1}^p$ is orthogonal with
    respect to the Frobenius inner product on $\frakv.$  For 
    all $z_j$ in $\calB,$  the Frobenius norm of $J_{z_j}$ is $
    \sqrt{-\trace J_{z_j}^2}= \sqrt{2r}.$
}\label{JZ-s}
\end{enumerate}
\end{theorem}

 To prove the theorem we require the lemma below.
\begin{lemma}\label{lin ind}
  Let $\frakn$ be a uniform Lie algebra of type $(p,q,r)$ with
uniform basis $\calB = \{ v_i \}_{i=1}^q \cup \{ z_j \}_{j=1}^p$
and degree $s.$ Let $v_i \in \calB.$ Then the set
\[ A_i = \{ [v_i, v_j] \, : \, \text{$v_j \in \calB$ and $[v_i,
v_j] \ne 0$} \}\] is a linearly independent subset of $\{ z_j
\}_{j=1}^p$ of cardinality $s.$
\end{lemma}

\begin{proof}  
  Let $v_i \in \calB,$ and let $A_i$ be as in the statement of the
lemma.  Since the degree $s$ in Part \eqref{s} of the definition of
uniform Lie algebra is assumed to be positive, $A_i$ is nonempty.

From Part \eqref{onek} of the definition of uniform Lie algebra we see
that $A_i \subseteq \{ \pm z_1, \pm z_2, \ldots, \pm z_p\}.$ As $\{
z_j \}_{j=1}^p$ is a subset of a basis, it is linearly independent.
Therefore, in order to show that $A_i$ is linearly independent, it
suffices to show that if $z_l \in A_i,$ then $-z_l \not \in A_i.$
Assume to the contrary that $z_l$ and $-z_l$ are both in $A_i.$ Then
there exist $v_j$ and $v_k$ in $\{ v_i \}_{i=1}^q,$ necessarily
distinct, so that $[v_i, v_j] = z_l$ and $[v_i, v_k] = -z_l.$ We then
have $[v_i, v_j] = -[v_i, v_k],$ and by Part \eqref{rdistinct} of the
definition of uniform Lie algebra, $v_j = v_k,$ a contradiction.
 Hence $z_l$ and $-z_l$
are not both in $A_i,$ and the set $A_i$ is linearly independent.

By Part \eqref{s} of   the definition of uniform Lie algebra, the
cardinality of $A_i$ is $s.$   
\end{proof}

Now we are ready to prove Theorem \ref{basic properties}.
\begin{proof}
Let $(\frakn, Q)$ be a uniform metric Lie algebra with
uniform basis $\calB = \{ v_i\}_{i=1}^q \cup \{ z_j\}_{j=1}^p.$ 
It follows easily from Part \eqref{onek} of Definition \ref{def of
uniform} that $\frakz = \myspan \{ z_j\}_{j=1}^p$ is contained 
in the center of $\frakn.$  To show that
the center is no larger, assume that $x$ is in the center.  Write
$x$ with respect to the uniform basis, so $x= \sum_{i=1}^q
\alpha_i v_i + \sum_{j=1}^p \beta_j z_j.$ We must show that
$\alpha_i = 0$ for all $i.$ Fix $i_0.$ By Part \eqref{rdistinct} of
the definition of uniform Lie algebra, there is some $v_k \in
\calB$ so that $[v_{i_0}, v_k ] \ne 0.$ Then
\[ 0 = [ v_k, x ] = \sum_{i=1}^q \alpha_l [v_k, v_i] + \sum_{j=1}^p
\beta_j [v_k, z_j] = \sum_{i=1}^q \alpha_i [v_k, v_{i}].\] Since
$[v_k, v_{i_0} ] \ne 0,$ Lemma \ref{lin ind} implies that
$\alpha_{i_0} = 0.$ Hence, the center of $\frakn$ is $\frakz.$

 Now we show that the commutator subalgebra 
$[\frakn, \frakn]$ is equal to $\frakz.$ It follows from Part
\eqref{onek} of the definition of uniform Lie algebra that $[\frakn,
\frakn] \subseteq \frakz.$ By Part \eqref{r} of the definition, $z_l
\in [\frakn, \frakn]$ for all $z_j \in \calB.$ Hence $[\frakn, \frakn]
= \frakz.$ From the properties of the bracket relations, we have
$[\frakn, [\frakn, \frakn]] = \{0 \},$ so $\frakn$ is a two-step
nilpotent Lie algebra.   

We omit the proof that $C(v_i) = \myspan ( \{ v_j \, : \, [v_i,
v_j] = 0\} \cup \{z_j\}_{j=1}^p),$ because it is similar to the
proof of Part \eqref{center}.  Since the spanning vectors are
linearly independent, the dimension of $C(v_i)$ is equal to the
cardinality of the spanning set.  By Part \eqref{s} of the
definition of uniform Lie algebra, $\{ v_j \, : \, [v_i, v_j] = 0\}
$ has cardinality $q-s.$ Hence the dimension of $C(v_i) $ is
$p+q-s.$

Part \eqref{s meaning} of the proposition follows immediately from
Lemma \ref{lin ind}.

For Part \eqref{r meaning} of the proposition, we use Proposition \ref{JZ}.
Fix $z_l.$ The definition of uniform Lie algebra says that there are
exactly $r$ disjoint pairs $\{v_i,v_j\}$ with $[v_i,v_j]=z_l.$ By
Proposition \ref{JZ}, the image of $J_{z_l}$ is the span of all
$v_i$ and $v_j$ so that $[v_i,v_j]=z_l.$ Since the cardinality of
that set is $2r,$ the map $J_{z_l}$ has rank $2r.$

Last we show that Part \eqref{JZ-s} holds. Let $z_k, z_l \in \calB.$
The number $\trace (J_{z_k} J_{z_l}^T)$ is nonzero if and only if
there are $i$ and $j$ so that $J_{z_k} v_i = \pm v_j$ and $ J_{z_l}
v_i = \pm v_j.$ But then $\pm z_k = [v_i, v_j] = \pm z_l.$ This is not
possible for distinct $k$ and $l,$ because $z_k$ and $z_l$ are
linearly independent.  By Proposition \ref{JZ}, after a change of
basis, $J_{z_l}$ is block diagonal, with nonzero blocks of form $
( \begin{smallmatrix} 0 & -1 \\ 1 & 0 \end{smallmatrix}).$ There are
precisely $r$ of these blocks. Hence $\trace (J_{z_l} J_{z_l}^T) = -
\trace (J_{z_l}^2) = 2r.$
\end{proof}

\begin{remark} 
 Statement \eqref{JZ-s} in the theorem may not be
improved to say that $ \trace J_{z}^2 = -2r$ for all unit $z
\in \frakz.$ For example, this fails in Example \ref{cyclic
algebra}.
\end{remark}

As a corollary to the previous theorem, since $p$ is
the dimension of the center, and $p + q$ equals the total
dimension, both of the parameters $p$ and $q$ are
isomorphism invariants for uniform Lie algebras.
\begin{cor}\label{invariants}
  Let $\frakn_1$ be a uniform Lie algebra of type $(p_1,q_1, r_1)$
and let $\frakn_2$ be a uniform Lie algebra of type $(p_2,q_2,
r_2).$ If $\frakn_1$ and $\frakn_2$ are isomorphic, then $p_1 =
p_2$ and $q_1 = q_2.$
\end{cor} 

The parameter $r$ is not an isomorphism
invariant for uniform Lie algebras, as seen in the following example taken
from \cite{nikolayevsky-preEinstein}.
\begin{example}\label{r  not invariant}
Let $\frakn = \frakh_3 \oplus \frakh_3$ with basis $\calB = \{x_1,
x_2, x_3, x_4\} \cup \{ z_1, z_2\},$ 
where $[x_1,x_2] = z_1$ and $[x_3,x_4] = z_2.$ 
The basis $\calB$ is a uniform basis of type $(2,4,1).$ 

 With respect to the   basis $\calC = \{u_i\}_{i=1}^4 
\cup \{w_j\}_{j=1}^2$ given by
\[u_1 = x_1 + x_3, u_2=x_2+x_4, u_3 = x_1 - x_3, u_4=x_2- x_4,
  w_1=z_1 + z_2, w_2= z_2-z_1,\]
 the Lie bracket is given by
\[ [u_1, u_2] = [u_3,u_4] = w_1, [u_2, u_3]=[u_4,u_1]=w_2. \]
Hence $\frakn = \frakh_3 \oplus \frakh_3$ is
 isomorphic to $\frakn^\prime(2r+2)$ with $r=2$ as in Example \ref{Damek Ricci gen}. 
With respect to the basis $\calC,$  $\frakn$ is uniform of type $(2,4,2).$ 
\end{example}

Two Lie algebras $\frakm$ and $\frakn$ with bases $\calB$ and
$\calC$ respectively are said to be {\em associates} with respect
to bases $\calB$ and $\calC$ if their structure constants
agree up to sign.  That is, 
 if $\alpha_{ij}^k$ denotes a structure
constant for $\frakm$ with respect to $\calB,$ and $\beta_{ij}^k$
denotes a structure constants for $\frakn$ with respect to $\calC,$
then $(\alpha_{ij}^k)^2 = (\beta_{ij}^k)^2$ for all $(i,j,k).$ The
Lie algebras in Examples \ref{quaternionic} and \ref{quaternion
associate} are nonisomorphic associates, as are the two families in
Example \ref{Damek Ricci gen}.  It is clear from the
definition that uniform Lie algebras come in classes of associates.
\begin{prop}\label{associates}  
Suppose that the Lie algebras $\frakm$ and $\frakn$ with bases
$\calB$ and $\calC$ respectively are  associates with respect
to bases $\calB$ and $\calC.$ Then $\frakm$ is a uniform Lie
algebra with respect to the uniform basis $\calB$ if and only if
$\frakn$ is a uniform Lie algebra with respect to the uniform basis
$\calC.$
\end{prop}

\begin{remark}\label{Y hat}
Given a family of associate uniform Lie algebras (relative to given
bases) as in Proposition \ref{associates}, how can we determine which
members of the family are isomorphic?  In general, it can be very difficult
to find an isomorphism between two Lie algebras.  However, there is a
simple computational method for finding isomorphic associate Lie
algebras when the two presentations
 are related by a change of basis which simply
changes signs of basis vectors.  This procedure is described in
Theorem B of \cite{payne-agag}; one simply needs to find the orbits of
a $\boldZ_2^n$ action on $\boldZ_2^m$ (defined in Definition 5 in
\cite{payne-agag}).  We will use this method to find these 
kinds of  isomorphic Lie
algebras within classes of associate uniform Lie algebras, but because
the set-up is somewhat technical, we do not reproduce the theorem
statement here and refer the reader to \cite{payne-agag} for details.
Worked out examples of the method are given in \cite{payne-methods};
see Examples 3.6, 3.7 and 4.5 there.
\end{remark}

When the parameter $r$ of a uniform Lie algebra with uniform basis
$\calB$ is equal to one, then it is isomorphic to all of its
associates relative to $\calB.$
\begin{prop}\label{r = 1}
Let $\frakn$ be a uniform Lie algebra of type 
$(p,q,1)$ with uniform basis $\calB.$  If $\frakm$ is an associate to
$\frakn$ with respect to the basis $\calB,$ then $\frakm$ and 
$\frakn$ are isomorphic.  
\end{prop}

\begin{proof}
Let $\alpha_{ij}^k$ denote structure constants relative to the
uniform basis $\{v_i\}_{i=1}^k \cup \{z_k\}_{k=1}^p,$ so $[v_i,
v_j] = \sum_{k=1}^p \alpha_{ij}^k z_k.$ Fix $k.$ But since $r=1,$
for all $k$ there is only one pair of indices $\{i,j\}$ so that 
structure constants $\alpha_{ij}^k$ and  $\alpha_{ji}^k$ are nonzero.  A change
of basis sending $z_k$ to $- z_k$ while leaving all other basis
vectors fixed changes the signs of$\alpha_{ij}^k$ and  $\alpha_{ji}^k$  while leaving all
remaining structure constants the same.  By making all possible
combinations of such changes of basis we obtain all possible sign
choices for the structure constants.
\end{proof}

Finally, we give some useful relationships the parameters $p, q, r$
and $s$ that we will use in Section \ref{classification}.
\begin{prop}\label{constraints}
  Let $\frakn$ be a 
uniform Lie algebra of type $(p, q,   r)$
with degree $s.$   Then 
\begin{enumerate}
\item{$2rp=sq,$}
\item{$s \le p \le rp =  \smallfrac{1}{2} sq \le (\begin{smallmatrix} q
      \\ 2 \end{smallmatrix}),$}\label{bounds 1}
\item{$2 \le 2r \le q.$}
\end{enumerate}
Let $G = (V,E)$ be a uniformly colored graph of type $(p, q,   r)$
with degree $s.$  
Then the same constraints hold, and $p$ divides $|E|.$
\end{prop}

\begin{proof}  We have previously remarked that $2rp=sq$
because both numbers are equal to the number of nontrivial
structure constants with respect to the uniform basis.

 By Lemma \ref{lin ind}, $s \le p.$
For each $v_i$ there are exactly $s$ basis vectors $v_j$ so
that $[v_i, v_j] \ne 0.$ The cardinality of $\{ v_j
\}_{j=1}^q$ is $q,$ and $[v_i, v_i]= 0,$ so $s \le q-1.$
Hence $\smallfrac{1}{2} sq \le \smallfrac{1}{2} (q-1)q =
(\begin{smallmatrix} q \\ 2 \end{smallmatrix}).$
Substituting $2rp$ for $sq$ gives $rp \le
(\begin{smallmatrix} q \\ 2 \end{smallmatrix}).$ 
Hence Part \eqref{bounds 1} of the proposition holds.

 Let $z_k$ be an element of the
uniform basis.  By Property \eqref{r} of a uniform Lie
algebra, there are $r$ disjoint pairs $\{ v_i, v_j \}$ drawn from 
$\{v_i\}_{i=1}^q,$ a set of cardinality $q,$ so
that $[v_i,v_j] =z_k.$ But then $q \ge 2r.$ Since $r \ge
1,$ we get $ q \ge 2.$
\end{proof}

\section{Uniformly colored graphs}\label{graphs}
 
\subsection{Definition and examples  
of uniformly colored graph}
Let $(V,E)$ be a graph with vertex set $V$ and edge set $E.$ We always
assume that all of our graphs are without loops or multiple edges. 
A graph is {\em regular of degree $s$} if each vertex has
exactly $s$ neighbors. An
{\em orientation} of the graph $(V,E)$ is a map $\sigma: E \to V
\times V$ assigning to each edge $\{ v,w\}$ one of the ordered pairs
$(v,w)$ and $(w,v).$ Then $(V,E),$ together with $\sigma,$ naturally
defines a directed graph.   

 An {\em edge coloring} of a directed or undirected graph $G =
(V,E)$ is a map $c$ from $E$ to a set $S.$ Let $G = (V,E)$ be a
graph with edge coloring $c: E \to S.$  A coloring is {\em proper}
if no adjacent edges have the same color. We say that a mapping $\phi
: G \to G$ is a {\em color-permuting automorphism} or {\em
automorphism of the colored graph} if $\phi$ is an automorphism of
the underlying graph, and there is a permutation $\sigma$
of the color set $S$ so that $c \circ \phi (v,w) = \sigma \circ c
(v, w)$ for all $(v,w) \in E.$ Two edge colorings $b$ and $c$ of a
graph are called {\em equivalent} if there is a color-permuting
automorphism $\phi$ so that $c = \phi \circ b.$

\begin{defn}\label{def of graph coloring} 
Let $(V,E)$ be an undirected graph with $q$ vertices which is regular
of degree $s.$ Let $S$ be a set of cardinality $p,$ and let the
surjective function $c: E \to S$ define an edge coloring of $(V,E).$
Then the edge coloring is a {\em uniform edge coloring} of type
$(p,q,r)$ with degree $s$ if the edge-colored graph has the following
properties.
\begin{enumerate}
\item{The coloring is proper.}\label{adjacent}
\item{Each color occurs the same number, $r,$ of times; i.e. 
 the cardinality
  of $c^{-1} (\{s\})$ is $r,$ for all $s \in S.$}\label{same number}
\end{enumerate}
When the undirected graph $(V,E)$ is endowed with such a coloring
it is called a {\em uniformly colored graph of type $(p,q,r).$} A
directed graph is a {\em uniformly colored graph of type $(p,q,r)$}
if it is an undirected uniformly colored graph of type $(p,q,r)$
endowed with an orientation.  A graph with a uniform edge coloring
is said to be {\em uniformly colored}.
\end{defn}

Note that if a uniformly colored graph $G=(V,E)$ is
directed, because $G$ arises from imposing an orientation on
an undirected graph, it is not possible for both
$(v_i,v_j)$ and $(v_i,v_j)$ to be in $E.$ For the
convenience of the reader, the next proposition summarizes
how the parameters $p,q,r$ and $s$ for a uniformly colored
graph are reflected in the graph.  The proof is an easy
application of definitions.
\begin{prop}\label{graph parameters}
A uniformly colored graph $G=(V,E)$ of type $(p,q,r)$ with degree $s$
has $q$ vertices and $rp$ edges.  The degree of the
underlying undirected regular graph is $s.$ There are $p$
colors, and each color occurs on exactly $r$ different edges.  
\end{prop}

A {\em decomposition} of a graph G is a set of subgraphs $H_1,...,
H_k$ which partition the edges of G.  Uniform edge colorings of
undirected graphs of type $(p,q,r)$ and degree $s$ are equivalent to
decompositions of $s$-regular graphs on $q$ vertices into $p$ copies
of the $r$-fold disjoint sum $K_2 + \cdots + K_2.$
\begin{prop}\label{decomposition}
  Let $G$ be an undirected uniformly colored graph of type $(p,q,r)$ with
degree $s.$ For $k=1, \ldots, p,$ let $H_k$ be the subgraph of $G$
consisting of edges colored with the $k$th color.  Then $H_1, H_2,
\ldots, H_p$ defines a decomposition of $G$ into $p$
disjoint subgraphs, all isomorphic to the disjoint sum $K_2 +
\cdots + K_2$ of $r$ copies of $K_2.$

Conversely, if $G= (V,E)$ is an undirected $s$-regular graph with $q$ vertices
and $H_1, \ldots, H_p$ is a decomposition of $G$ into disjoint
subgraphs each isomorphic to the disjoint sum $K_2 + \cdots + K_2,$
then the coloring $c: E \to [p]$ defined by coloring an edge $k$ if
it is in the factor $H_k$ is a uniform edge coloring of $G.$
\end{prop}  
 (We use $[n]$ to denote the set $\{1, 2, \ldots, n \}.$)
 The proposition follows from the definition of uniformly colored
graph.  The next example shows that any regular graph admits a uniform
edge coloring.
\begin{example}\label{trivial}
Let $(V,E)$ be a regular graph with $q$ vertices and $p$ edges.  Let
the identity map $id: E \to E$ define a coloring.  The resulting
edge-colored graph is uniformly colored of type $(p,q,1).$  
\end{example}

\subsection{Graph-algebra correspondences}

We associate to any uniform Lie algebra a uniform 
graph, and to any uniform graph a uniform Lie algebra.

\begin{defn}\label{algebra to graph}
Let $(\frakn,Q)$ be a uniform Lie algebra of type $(p,q,r)$ with
uniform basis $\calB = \{v_i\}_{i=1}^q \cup \{ z_j\}_{j=1}^p.$ Define
a set of vertices by $V = \{v_i\}_{i=1}^q,$ and let the set of
directed edges be
\[ E = \{ (v_i, v_j) \, : \, \text{$[v_i, v_j] = z_k$ for some $k$}
\}.\] 
Define an edge coloring of the graph $(V,E)$ by mapping $E$ into
the set $\{ z_j\}_{j=1}^p$ so that the edge $(v_i, v_j) \in E$ is
assigned to $z_k$ if $[v_i, v_j] = z_k.$
\end{defn}
By Property \eqref{rdistinct} of
the definition of uniform metric Lie algebras, the coloring function is
well-defined. 

\begin{prop}\label{graph properties}  
 The graph associated to a uniform Lie algebra of type $(p,q,r)$ with
degree $s,$ as in Definition \ref{algebra to graph}, is a uniformly
colored graph of type $(p,q,r)$ with degree $s.$ 
\end{prop}

\begin{proof} 
The graph has no loops because $[v_i,v_i] = 0$ for all $i.$ There are
no multiple edges from $v_i$ to $v_j$ because $[v_i, v_j]$ is
single-valued, and it is not possible to have both $(v_i,v_j)$ and
$(v_j,v_i)$ as an edge because $[v_i,v_j] = [v_j, v_i]$ implies that
$[v_i,v_j] = 0.$ Clearly there are $q$ vertices and $p$ colors.  By
Property \eqref{rdistinct} of Definition \ref{def of uniform},
the coloring is proper.  Property \eqref{r} of
Definition \ref{def of uniform} insures that each color occurs the
exactly $r$ times and that the coloring function is surjective.
\end{proof}
Now we associate a uniform Lie algebra to any uniformly colored graph.
\begin{defn}\label{graph to algebra}
Let $(V,E)$ be a uniformly colored directed graph of type $(p,q,r)$
and degree $s.$ Let $V = \{v_i\}_{i=1}^q,$  let $S = \{
z_j\}_{j=1}^{p}$ and let $c: E \to S$ be the surjective edge
coloring function.  Let $\frakn$ be the vector space with basis $V
\cup S.$ Define a Lie bracket on $\frakn$ by setting
\begin{equation}\label{bracket}  
[v_i, v_j]  = 
\begin{cases}  
 c(v_i, v_j)  & \text{if $(v_i,v_j) \in E$} \\ 
-c(v_i, v_j)  & \text{if $(v_j,v_i) \in E$} \\
  0   & \text{otherwise} \\
\end{cases},\end{equation}
setting $[v_i, z_j] = 0$ and $[z_j, z_k] = 0$ for all  $i,j,k,$ and 
extending bilinearly.  
\end{defn}
As the graph is uniform, it is not possible for both $(v_i,v_j)$ and
$(v_j,v_i)$ to be in $E,$ so the Lie bracket in the definition is
well-defined.  By definition the Lie bracket is skew-symmetric.  It is
easy to see that $[\frakn,\frakn] = \myspan \{z_j\}_{j=1}^p$ and
$[\frakn, [\frakn,\frakn] ] = 0,$ so the algebra is two-step
nilpotent. Hence, the Jacobi Identity holds trivially.  Thus, the
product in Equation \eqref{bracket} truly does define a Lie algebra.

\begin{example}\label{cycle graph}
The uniformly colored graph affiliated with the cyclic Lie algebra
$\frakm(q)$ in Example \ref{cyclic algebra} is the cycle graph $C_q$
with $q$ vertices (endowed with an orientation).  All $q$ edges have
different colors.
\end{example}

The correspondences we have defined between uniformly colored graphs
and uniform Lie algebras are inverse to one another.
\begin{prop}\label{graph to algebra prop}  
The correspondences between uniform Lie algebras (with fixed uniform
bases) and uniformly colored graphs defined in Definitions \ref{algebra
to graph} and Definition \ref{graph to algebra} are inverse to
one another.  The Lie algebra  corresponding to a uniformly
colored graph of type $(p,q,r)$ with degree $s$ 
 is a uniform Lie algebra
with uniform basis  of type $(p,q,r)$ and degree
$s.$ 
\end{prop}
The proof is straightforward, so we leave it to the reader. 

\begin{remark}\label{more general} 
  The correspondences in Definitions \ref{algebra to graph} and
\ref{graph to algebra} between edge-colored graphs and algebras are
defined more generally, even if not all of the properties of a
uniform Lie algebra or uniformly colored graph hold.  In the most
general situation, there is a correspondence between edge-weighted
directed graphs (possibly with loops or multiple edges) and
two-step nilpotent Lie algebras.  This is carefully described in
Remark 3.2 of \cite{ray-16}.
\end{remark}

\begin{remark}\label{Dani-Mainkar} 
The two-step nilpotent Lie algebras of {\em graph type} from
\cite{dani-mainkar-05} can be defined as follows.  Let $G= (V,E)$ be a
 graph on $q$ vertices without multiple edges or loops, where
$V = \{v_i\}_{i=1}^q.$ Let $\frakv$ be the real span of $V,$ so
$\frakv \cong \boldR^q.$ Let $\frakn = \frakv \oplus \frakz,$ where
$\frakz = \myspan \{ v_i \wedge v_j \, : \, \{v_i, v_j\} \in E \}
\subseteq \frakso(q),$ and define a Lie bracket for $\frakn$ so that
the only nonzero bracket relations of basis vectors are $[v_i, v_j] =
v_i \wedge v_j$ if the edge $\{v_i, v_j\} \in E.$    
The Lie algebras in Examples \ref{free} and \ref{cyclic algebra} are
of graph type.  

 A Lie algebra of graph type need not come from a regular graph, so
not all Lie algebras of graph type are uniform.  Lie algebras of graph
type correspond, in the sense of Proposition \ref{graph to algebra},
to colored graphs with each edge colored a different color, as in
Example \ref{trivial}; hence if they are uniform, the parameter $r$ is
equal to one as in Proposition \ref{r = 1}. Therefore  uniformly
colored graphs of type $(p,q,r)$ with $r \ne 1$ may fail to be of 
graph type.
\end{remark}

\begin{remark}\label{ideals} 
 An alternate point of view on the definitions of uniform Lie algebra
and Lie algebras of graph type is from the perspective of ideals.
Every two-step nilpotent Lie algebra with $n$ generators is the
quotient of $\frakf_{n,2},$ the free two-step nilpotent Lie algebra on
$n$ generators, by a central ideal.  Let $x_1, x_2, \ldots, x_n$
denote $n$ generators of $\frakf_{n,2}$ and let $\calB = \{ x_i
\}_{i=1}^{n} \cup \{ x_i \wedge x_j \}_{1 \le i < j \le n}$ be the
basis for $\frakf_{n,2}$ as in Example \ref{free}.  Lie algebras
of graph type are quotients of free two-step nilpotent Lie algebras 
 by monomially generated  ideals, i.e, ideals
spanned by subsets of $\{ x_i \wedge x_j \}_{1 \le i < j \le n},$ and
every monomially generated  ideal defines a Lie algebra of graph type.  In
contrast, akin to toric ideals, ideals defining uniform Lie algebras
are generated by elements of the form $x_i \wedge x_j,$ with $i \ne j,$
or $x_i \wedge x_j \pm x_k \wedge x_l,$ where $i,j,k,l$ are distinct.
\end{remark}

\subsection{Translations between 
algebraic properties and graph properties}
The $J_z$ maps defined in Section \ref{properties} completely encode
the algebraic structure of a metric nilpotent Lie algebra.  The {\em
skew-adjacency matrix} of a directed graph $G$ is defined to be the
matrix $A(G) = (a_{ij})$ with $a_{ij} = 1$ if $(v_i, v_j) \in E,$
$a_{ij} = -1$ if $(v_j, v_i) \in E,$ and $a_{ij} = 0$ otherwise.
Skew-adjacency matrices for uniformly colored graphs are closely
related to the $J_z$ maps for the corresponding uniform nilpotent Lie
algebra.
\begin{prop}\label{J}
Let $\frakn$ be a uniform metric Lie algebra with uniform basis
$\{v_i\}_{i=1}^q \cup \{z_j\}_{j=1}^p$ and let $G = (V,E)$ be the
associated uniformly colored graph with coloring function $c : E \to
[p].$ Then $ J_{z_j} = -A(H_j),$ where $H_j$ is the
subgraph consisting of the edges with color $j,$ so for $z =
\sum_{j=1}^p b_j z_j,$ the endomorphism $J_z$ is given by $J_z =
-J\left( \sum_{j=1}^p b_j z_j\right) = -\sum_{j=1}^p b_j \, A(H_j).$
\end{prop}
To prove the proposition, use Proposition \ref{JZ}.

In geometry, it is of interest to 
understand when submanifolds of homogeneous
manifolds are totally geodesic.  Totally geodesic subalgebras of
metric nilpotent Lie algebras are tangent to totally geodesic
submanifolds of the corresponding nilmanifolds (See
\cite{eberleinnilgeom1, cairns-et-al-13}).  A subalgebra $\frakm$ of a
metric nilpotent Lie algebra is {\em totally geodesic} if it is
invariant under the $J_z$ map for all $z \in \frakm.$ Recall that the
uniformly colored graph $G = (V,E)$ corresponding to the Lie algebra
$\frakn$ with uniform basis $\calB =\{ v_i\}_{i=1}^q \cup
\{z_j\}_{j=1}^p$ has vertex set  $V = \{ v_i\}_{i=1}^q$ and edge 
colors in the set $\{z_j\}_{j=1}^p.$
\begin{prop}\label{totally geodesic}
Let $(\frakn, Q)$ be a uniform metric nilpotent Lie algebra with
uniform basis $\calB = \{ v_i\}_{i=1}^q \cup \{z_j\}_{j=1}^p.$ Let
$G = (V,E)$ be the corresponding uniformly colored graph. Let $V'
\subseteq V$ be a  subset of $V$   and let $S' $ be a
 subset of $\{z_j\}_{j=1}^p,$ where $V' \cup S' \ne \emptyset.$ 
Let $\frakm$ be the subspace of $\frakn$ spanned by $ V' \cup S'.$
\begin{enumerate}
\item{The subspace $\frakm$ is a subalgebra of $(\frakn, Q)$ if and
only if whenever $v_i$ and $v_j$ are in $V'$ and $(v_i, v_j) \in E,$
the color $c(v_i,v_j)$ is in $S'.$ }
\item{Suppose $S' \ne \emptyset.$
The subspace $\frakm$ is invariant under the $J_z$ map for all
$z \in \frakm$ if and only if for all $z_k \in S',$ every edge with
color $z_k$ has neither or both of its
vertices in $V'.$ }
\end{enumerate}
\end{prop}

\begin{proof}  
The subspace $\frakm$ is a subalgebra if and only if 
$[v_i,v_j] \in \myspan S'$ for all $v_i, v_j \in V'.$ But $[v_i,v_j] $
is nontrivial if and only if $(v_i, v_j) \in E$ or $(v_j, v_i) \in E,$ 
and in that case $[v_i,v_j] = \pm z_k,$ where $z_k$ is the color of
the edge between $v_i$ and $v_j.$ 

By Proposition \ref{J},  the $J_{z_k}$ map for color
$z_k$ inverts edges with color $z_k;$ i.e., if $(v_i, v_j) \in E$ and
$c(v_i,v_j) = z_k,$ then $J_{z_k}(v_i) = v_j$ and $J_{z_k}(v_j) =
-v_i,$ and
if there is no edge colored $z_k$ incident to vertex $v_i,$ then 
$J_{z_k}(v_i) =0.$  Hence, if $z_k \in S',$
$\frakm$ is $J_{z_k}$-invariant if and only if any $z_k$-colored edge
with one vertex in $V'$ has the other vertex in $V'.$
\end{proof}

The next example describes a type of totally geodesic subalgebra for
a uniform metric nilpotent Lie algebra $(\frakn, Q).$
\begin{example}  
Let $H$ be the union of connected components of $G$ and
let $V'=V_H$ be the set of vertices spanning $H.$ Let $S'=S_H$ be the
set of colors assigned to edges of $H$.  Then $V_H \cup S_H$ spans a
totally geodesic subalgebra of $(\frakn, Q).$
\end{example}

Recall that the basis for defining the Lie algebra $\frakn$
corresponding to a uniformly colored graph $G = (V,E)$ with colors
in $S$ is $\calB = V \cup S.$ Any color-permuting
automorphism $\phi$ of $G$ with permutation $\sigma$ of $S$
extends to a function $\hat \phi: \frakn
\to \frakn$ of the corresponding uniform Lie algebra $\frakn,$ by
defining  $\hat \phi(v) = \phi (v)$
for $v \in V$ and $\hat \phi (z) = \sigma (z)$ for $z \in S,$ and
extending $\phi$ from $\calB$ bilinearly.
We show that the function $\hat \phi$ is an automorphism of
$\frakn.$
\begin{prop}\label{symmetry}  
Let $\phi$ be a color-permuting automorphism of a uniformly colored
directed graph $G$ with corresponding uniformly colored Lie algebra
$\frakn.$ Then the map $\hat \phi : \frakn \to \frakn$ is an
automorphism of $\frakn.$  
Furthermore, the map $\phi \mapsto \hat \phi$ from  group of
color-permuting automorphisms of $G$ to the
automorphism group of $\frakn$
is injective.
\end{prop}

\begin{proof} 
Let  $G = (V,E)$ be a uniformly colored graph
 with coloring function $c:E \to S, $ and denote the
vertices of $G$ by $V = \{v_i\}_{i=1}^q.$ Let $\phi: G \to G$ and
$\sigma: S \to S$ define a color-permuting automorphism of $G.$
Recall that $[v_i,v_j] \ne 0$ if and only if $(v_i,v_j) \in E$ or
$(v_j,v_i) \in E.$ Hence $[v_i,v_j] \ne 0$ if and only if $[\hat \phi
(v_i), \hat \phi (v_j)] \ne 0,$ and we need only to check that $[\hat
\phi (v_i) , \hat \phi (v_j)] = \hat \phi ([v_i, v_j])$ for $(v_i,
v_j) \in E.$ Assume that $(v_i, v_j) \in E.$ Then $(\phi(v_i),
\phi(v_j)) \in E.$ By definitions of the Lie bracket and $\hat \phi,$
\[ [\hat \phi (v_i) , \hat \phi (v_j)] = [\phi(v_i), \phi(v_j)] =
c((\phi(v_i), \phi(v_j))) = c( ( \phi (v_i,v_j) ),\] while
\[\hat \phi ([v_i, v_j]) = \hat \phi (c(v_i, v_j) ) = \sigma (c(v_i,
v_j) ).\] Thus, $[\hat \phi (v_i) , \hat \phi (v_j)] = \hat \phi
([v_i, v_j])$ as desired.  Clearly the map $\phi \mapsto \hat \phi$ is
injective.
\end{proof}

Recall that the disjoint union 
$G_1 + G_2= (V_1 \cup V_2, E_1 \cup E_2)$ of graphs $G_1 =
(V_1, E_1)$ and $G_2= (V_2, E_2)$ is the graph whose vertex
set is the union of the vertex sets of $G_1$ and $G_2,$ and
whose edge set is the union of the edge sets of $G_1$ and
$G_2.$ If $G_1$ and $G_2$ are edge-colored by coloring functions
$c_1 : E_1 \to S_1$ and $c_2 : E_2 \to S_2,$ then it is
natural to color $G_1 + G_2$ by defining  $c: E_1 \cup
E_2 \to S_1 \cup S_2$ by $c(e) = c_1(e)$ if $e \in E_1,$ and $c(e) =
c_2(e)$ if $e \in E_2.$

In the next proposition we show how the disjoint union of two
edge-colored directed graphs translates at the level of corresponding
uniform Lie algebras.  Actually, we do not need to assume that the
graphs or Lie algebras in the proposition are uniform; even if
regularity axioms involving $p, q,$ and $r$ are dropped, Definition
\ref{def of uniform} and Definition \ref{def of graph coloring} still
make sense, and Propositions \ref{algebra to graph} and \ref{graph to
algebra} still hold.  (See Remark \ref{more general}.)
\begin{prop}\label{union equals}
  Suppose that for $k=1,2,$ $G_k = (V_k,E_k)$ is a colored graph of
type $(p_k,q_k,r_k)$ with coloring function $c_k: E_k \to S_k. $ Let
$\frakn_1$ and $\frakn_2$ be the corresponding Lie algebras, and let
$\calB_1 = \{v_i\}_{i=1}^{q_1} \cup \{z_j\}_{j=1}^{p_1}$ and $\calB_2
= \{x_i\}_{i=1}^{q_2} \cup \{w_j\}_{j=1}^{p_2}$ be the corresponding
uniform bases of $\frakn_1$ and $\frakn_2$ respectively.  Let $\frakn$
be the Lie algebra associated to the colored graph $G_1 + G_2.$ Let
$\fraki$ be the ideal in $\frakn_1 \oplus \frakn_2$ defined by
\[ \fraki = \myspan \{ (z_i, - w_j) \, : z_i = w_j \}\] if $S_1 \cap
S_2 \ne \emptyset$ and let $\fraki = \{ (0,0)\}$ otherwise.  Then
\begin{equation}\label{cat def}
\frakn (G_1 + G_2) :=  (\frakn_1 \oplus
\frakn_2) / \fraki.\end{equation} 
In particular, if $S_1$ and
$S_2$ are disjoint, then $\frakn \cong \frakn_1 \oplus \frakn_2.$
\end{prop}

The idea of the proof is as follows.  The center of $\frakn_1 \oplus
\frakn_2$
is spanned by elements $(z_i,0)$ and $(0,w_j)$ with $z_i$ in the first
color set $S_1$
and $w_j$ in the second color set $S_2.$  To get $\frakn (G_1 + G_2)$,
for every color $z_i = w_j$
occurring in both $S_1$ and $S_2,$  we need to identify  $(z_i,0)$ and
$(0,w_j).$  Hence the ideal $\fraki$ must be spanned by $(z_i,-w_j),$
for all colors with $z_i = w_j.$

 In the case
that $S_1 \supseteq S_2,$ the Lie algebra $\frakn (G_1 + G_2)$
 in the proposition is defined by a process that Jablonski
calls {\em concatenation} of the ``structure matrices'' of $\frakn_1$
and $\frakn_2$ (Section 3, \cite{jablonski-moduli}).  In the case 
that $S_1 \supseteq S_2,$ the uniform basis for 
$\frakn (G_1 + G_2)$ may be
naturally identified with $\{ v_i\}_{i=1}^{q_1}\cup \{z_j\}_{j=1}^p.$

We would like to know when the  graph union of
uniformly colored graphs defines a uniformly colored graph.
In the next proposition we consider the case that the coloring sets
are disjoint.  

\begin{prop}\label{graph union disjoint color set} 
For $k=1$ and $2,$ let $G_k= (V_k, E_k)$ be a uniformly
colored graph of type $(p_k,q_k,r_k)$ with degree $s_k,$
and let $c_k: E_k \to S_k$ be the coloring function for
$(V_k, E_k).$ Assume that the graphs $G_1$ and $G_2$ are
disjoint, and the coloring sets $S_1$ and $S_2$ are
disjoint.  Define the coloring function $c : E_1 \cup E_2
\to S_1 \cup S_2$ by $c(e) = c_1(e)$ if $e \in E_1,$ and
$c(e) = c_2(e)$ if $e \in E_2.$ Then the following are
equivalent:
\begin{enumerate}
\item{$G_1 + G_2,$ endowed with the coloring function $c,$ is a
uniformly colored graph.}
\item{$s_1 = s_2$ and $r_1 = r_2.$}
\end{enumerate}
Furthermore, if $G_1 + G_2,$ endowed with the coloring function $c,$ is
uniformly colored, then  it is of type $(p_1 + p_2, q_1 + q_2, r),$ where
$r=r_1=r_2,$ and it has degree $s = s_1 = s_2.$
\end{prop}

The proof is elementary.  We leave it to the reader.  



If we use Proposition \ref{union equals} to re-interpret
Proposition \ref{graph union disjoint color set} in terms of
algebras, we see how uniformly colored Lie algebras behave under
direct sums.
\begin{cor}\label{sum}
Let $\frakn_1$ be a uniform Lie algebra of type $(p_1, q_1, r_1),$
with uniform basis $\calB_1,$ and let $\frakn_2$ be a uniform Lie
algebra of type $(p_2, q_2, r_2),$ with uniform basis $\calB_2.$
Let $\frakn = \frakn_1 \oplus \frakn_2$ be the direct sum, and let
$\calB$ be the union of $i_1(\calB_1)$ and $i_2(\calB_2)$ where
$i_j : \frakn_j \to \frakn_1 \oplus \frakn_2,$ for $i=1, 2,$ are
the natural inclusion maps.  Then $\calB$ is a uniform basis of
type $(p,q,r)$ if and only if $r_1 = r_2$ and $s_1 = s_2.$ If
$\calB$ is a uniform basis, then $\frakn$ is uniform of type $(p_1
+ p_2, q_1 + q_2, r),$ where $r = r_1 = r_2.$
\end{cor}

Next we look at the union of disjoint uniformly colored
graphs having same coloring set to see when we get a uniformly
colored graph.   We leave the proof to the reader.  
\begin{prop}\label{graph union same color set} 
For $k=1,2,$ let $G_k= (V_k, E_k)$ be a uniformly colored graph of
type $(p,q_k,r_k)$ with degree $s_k,$ and edge coloring function
 $c_k: E_k \to S.$ 
 Suppose that the graphs $G_1$ and $G_2$
are disjoint.  Define $c: E_1 \cup E_2 \to S$ by $c(e) = c_1(e)$ if $e
\in E_1,$ and $c(e) = c_2(e)$ if $e \in E_2.$ Then the following are
equivalent:
\begin{enumerate}
\item{$G_1 + G_2,$ endowed with the coloring function $c,$ is a
uniformly colored graph.}
\item{$s_1 = s_2.$ }
\end{enumerate}
Furthermore, if $G_1 + G_2$ endowed with the coloring function $c$
is uniformly colored, then it is of type $(p, q_1 + q_2, r_1+r_2),$
with degree $s = s_1 = s_2.$
\end{prop}

Note that if $G$ is a uniformly colored graph, it is not true that 
all of its connected components must be uniformly colored. 


Now we translate Proposition \ref{graph union same color
set} from graphs to algebras.  This corollary overlaps with
 Proposition 3.4 of \cite{jablonski-moduli}, which says
that if $\frakn_1 = \frakv_1 \oplus \frakz$ and $\frakn_2 =
\frakv_2 \oplus \frakz$ have Einstein extensions, then so
does a concatenation of 
$\frakn_1$ and $\frakn_2.$ 
\begin{cor}\label{concatenations} 
Let $\frakn_1$ be a uniform Lie algebra of type $(p, q_1, r_1)$ of
degree $s$ with uniform basis $\calB_1 =\{ v_i\}_{i=1}^{q_1} \cup
\{z_j\}_{j=1}^p,$ and let $\frakn_2$ be a uniform Lie algebra of type
$(p, q_2, r_2)$ of degree $s$ with uniform basis $\calB_2 =\{
x_i\}_{i=1}^{q_1} \cup \{z_j\}_{j=1}^p.$
Let $\frakn$ be the concatenation of $\frakn_1$
and $\frakn_2$ as in Equation \eqref{cat def}.  Then $\frakn$ is
uniform of type $(p, q_1 + q_2, r_1 + r_2).$ 
\end{cor}


\section{Examples of uniform Lie algebras}\label{examples}
\subsection{Examples defined by
colorings of notable graphs}
Many well-known graphs from combinatorial, geometric or
number-theoretic constructions admit uniform edge colorings.  We give
just one example here.  Recall that for the Kneser graph $K_{n,m}=
(V,E),$ the vertex set $V$ is the set of subsets of $[n]$ having
cardinality $m.$ Two vertices (sets) are connected by an edge if and
only if they are disjoint.  The graph is regular with degree
$(\begin{smallmatrix} n -m \\ m \end{smallmatrix}).$ From the
 symmetry of the roles of the vertices, any permutation of the $n$ 
vertices is an automorphism.
\begin{prop}\label{kneser graph}
Let $K_{n,m}= (V,E)$ be a Kneser graph, where $m < n/2$.
Then $K_{n,m}$ admits a uniform edge coloring.  With this
coloring, $K_{n,m}$ is a uniformly colored graph of type
$(p,q,r) = \left( (\begin{smallmatrix} n \\
2m \end{smallmatrix}), (\begin{smallmatrix} n \\
m \end{smallmatrix}) , \smallfrac{1}{2}
(\begin{smallmatrix} 2m \\ m \end{smallmatrix}) \right).$
\end{prop}

\begin{proof}
 First we define the coloring function $c.$ Two adjacent vertices
have set union of cardinality $2m;$ we color the edge between them
with the set $[n] \setminus (S_1 \cup S_2).$ Such a set is a subset
of $[n]$ of cardinality $(\begin{smallmatrix} n \\ n -
2m \end{smallmatrix}) = (\begin{smallmatrix} n \\
2m \end{smallmatrix}),$ so there are $p = (\begin{smallmatrix} n \\
2m \end{smallmatrix})$ colors.  Clearly, from the $S_n$ symmetry,
the coloring map is surjective.

This coloring has the property that if the edge $e$ is incident to
vertices $v$ and $w,$ then the subsets $v, w$ and $c(e)$ of $[n]$ form a
partition of $n.$  Hence, for each edge,  the edge's color and one of the vertices it
is incident to determine the other vertex incident to the edge.   
Hence, no vertex is incident to two edges with the same color, and
 the coloring is proper.  Each color occurs the same
number of times and each vertex has the same degree, again due to
the $S_n$ symmetry of the graph.

The degree $s$ of vertex $v$ is the number of subsets of $[n]$
disjoint from $s,$ which is 
$(\begin{smallmatrix}  n-m  \\  m  \end{smallmatrix}).$
We compute $r$ using $r = \frac{sq}{2p} = \smallfrac{1}{2} 
(\begin{smallmatrix}  2m  \\  m  \end{smallmatrix}).$
\end{proof}

As a special case of Proposition \ref{kneser graph}, we get a
uniform edge coloring of the Petersen graph $K_{5,2}.$ After 
assignments of orientation we may apply Proposition \ref{graph to
algebra prop} to define uniform Lie algebras.

\begin{example}  A uniform Lie algebra of type $(5,10,3)$ defined
by the uniform edge coloring of the Petersen graph $K_{5,2}$ has
basis $\calB = \{ v_{ij}\}_{1 \le i < j \le 5} \cup \{ z_k
\}_{k=1}^k$ and Lie brackets determined by relations of form
$[v_{ij}, v_{kl}] = \pm z_m$ for all $i,j,k,l,m$ with $i< j, k<l$ and
$\{i,j,k,l,m\}=[5].$
\end{example}

\subsection{Examples defined by Cayley graphs}
Let $G$ be a nontrivial group, and let $T \subseteq G$ be a
nonempty set of elements of $G$ all of order 2.  The set
$T$ need not be a generating set.  Let $\Cay(G,T) = (G,E)$
be the Cayley graph of $G$ relative to $T;$ it is a digraph
with vertex set $G$ and arc set $\{ (g,tg) \, : \, g \in
G, t \in T\}.$ Since the elements in $T$ all have order
two, for all $t \in T$ and all $g \in G,$ both $(g, tg)$
and $(tg, g)$ are edges. We identify such pairs with an
undirected edge $\{g, tg\}$ to view $\Cay(G,T)$ as an undirected
graph.  Because the identity of $G$ is not in $T,$
$\Cay(G,T)$ has no loops or multiple edges.  Define a
coloring function $c: E \to T$ so that if $g \in G$ and $t
\in T,$ the edge $\{g, tg\}$ is assigned the color $t.$
\begin{prop}\label{Cayley} 
Let $G$ be a nontrivial group of cardinality $q,$ and let
$T \subseteq G$ be a nonempty set of $p \ge 1$ elements of
$G,$ all of order 2.  Let $\Cay(G,T)$ be the undirected
Cayley graph of $G$ relative to $T.$ The function $c$
defined above defines a uniform edge coloring on
$\Cay(G,T)$ with respect to which $\Cay(G,T)$ is a
uniformly colored graph of type $(p,q,q/2)$ and degree $p.$
\end{prop}

\begin{proof} Clearly there are $q$ vertices.
Fix a vertex $g$ and $t \in T.$ Since $t$ has order two,
$tg \ne g,$ and there is an edge labelled $t$ between $t$
and $tg.$ That makes a total of $p$ edges incident to $g.$
Hence, $\Cay(G,T)$ is regular of degree $s =
p,$ and  the coloring map is surjective.  Because
$G$ is a group, there is exactly one edge colored $t$ that
is incident to $g.$ This gives a total of $r = q/2$ edges
labelled $t.$ 
\end{proof}

\begin{example}\label{Klein} 
If we let $G$ be the additive group $\boldZ_2 \times
\boldZ_2$ and let $T_1 = \{ (1,0), (0,1), (1,1)\},$ then
the resulting uniformly colored complete graph is the one
corresponding to the uniform Lie algebras in Examples
\ref{quaternionic} and \ref{quaternion associate}.  If we
instead take $T_2 = \{ (1,0), (0,1)\},$ then we get the
uniformly colored cycle graph associated to Example \ref{Damek
Ricci}.  Finally, the set $T_2 = \{ (1,0)\}$ gives the
graph $K_2 + K_2$ colored with one edge color; the
corresponding Lie algebra is the five-dimensional
Heisenberg algebra $\frakh_5$ as in Example
\ref{Heisenberg}.
\end{example}
We may let $G$ be any finite reflection group.
\begin{example}\label{S3}
If we let $G = S_3$ and $T = \{ (1 \, 2), (1 \, 3), (2 \,
3)\},$ then $\Cay(G,T)$ is the Thomsen graph and the
coloring described in Proposition \ref{Cayley} makes
$\Cay(G,T)$ into a uniform Lie algebra of type $(3,6,3).$
The corresponding nine-dimensional uniform Lie algebra is
of type $(3,6,3).$ Note that the Thomsen graph is
isomorphic to the complete bipartite graph $K_{3,3},$
and that elements of odd order are in one partite set  
of $\Cay(G,T)$  while 
even  elements are in the other.  

We leave it to the reader to confirm that, more generally,
if $G$ is the dihedral group of order $2p,$ by taking $T$
to be the set of all reflections, we get a uniform edge
coloring on the complete bipartite graph $K_{p,p}.$ The
corresponding $3p$-dimensional uniform Lie algebra is of
type $(p,2p,p).$  
\end{example}

\subsection{Examples defined by one-factorizations and
near-one-factorizations}\label{1-factorizations}

Let $G= (V,E)$ be a graph.  A {\em factor} of $G$ is a subgraph with
vertex set $V$ and edge set $E' \subseteq E.$ A {\em factorization} of
$G$ is a set of factors of $G$ so that the factors are pairwise
edge-disjoint and whose union is $G.$ A {\em one-factor} is a factor
which is a regular graph of degree one.  A {\em one-factorization} of
$G$ is a factorization for which each factor is a one-factor.  (If $G$ is
oriented, these definitions extend in the obvious way.)  Clearly, a
graph on an odd number of vertices can not have a one-factorization.
In this case, the closest thing to a one-factorization is a {\em
near-one-factorization,} defined as follows.  A set of edges which
covers all but one vertex in a graph is called a {\em
near-one-factor}.  A decomposition of a graph as the union of
near-one-factors is a {\em near-one-factorization.}  We may use 
Proposition \ref{decomposition} to define a uniform edge coloring 
of a graph admitting a one-factorization or near-one factorization.
Hence, every one-factorization or near-one-factorization of a graph
defines a class of associate uniform metric Lie algebras. 

\begin{prop}\label{1-factorization}  
Let $G$ be an  $s$-regular graph $G$ with $q$
vertices and $m$ edges.   
\begin{enumerate}
\item{If $q$ is even, every one-factorization of $G$ defines a
uniformly colored graph of type $(s, q, q/2),$ and $m= sq/2.$ }
\item{If $q$ is odd, every near-one-factorization of $G$ defines a
uniformly colored graph of type $(\smallfrac{sq}{q-1}, q,
\smallfrac{q-1}{2}),$ and $m= sq/2.$}
\end{enumerate}
\end{prop}

\begin{proof}
Proposition \ref{decomposition} defines a uniform edge coloring for a
one-factorization or near-one factorization of a graph.  If $q$ is
even and $G$ has a one-factorization, each factor has $q$ vertices and
$q/2$ edges, so each color occurs $r=q/2$ times in such the
corresponding coloring.  Substituting into $2rp=sq$ gives $p=s.$ The
number of edges is $m = rp = sq/2.$ If $q$ is odd, and $G$ has a
near-one-factorization, each graph in the decomposition has $(q-1)/2$
edges, so each color occurs $r = (q-1)/2$ times in the associated
coloring.  Substituting into $2rp=sq$ gives $p=sq/(q-1).$ Then $m= rp=
sq/2.$
\end{proof}

Let $K_n$ denote the complete graph on $n$ vertices.   All complete graphs
admit a one-factorization or a near one-factorization, depending on
whether their order is even or odd. 
Up to equivalence $K_5$ has a unique near-one-factorization.  
This should be known, but we could not find a reference.  
See \cite{schroeder-thesis} for a detailed proof. Alternately, use the
fact that by adding a point at infinity, every near-one-factorization
of $K_5$ extends to a one-factorization of $K_6,$ and apply
Sylvester's argument for the uniqueness of one-factorizations of $K_6$
(up to isomorphism) to get a canonical presentation of 
near-one-factorizations of $K_5$  (see \cite{cameron-onefact}).
\begin{example}\label{K5 near-one-factorization}
Let $K_5= (V,E)$ be the complete graph on 5 vertices
 in  $\{v_i\}_{i=1}^5.$    Define the  edge coloring
$c : E \to [5]$ by  
\begin{align*}
c( \{e_2,e_5\})  &= c( \{e_3,e_4\}  )= 1 \\ 
c( \{e_4,e_5\} ) &= c( \{e_1,e_3\} ) = 2 \\ 
c( \{e_1,e_5\} ) &= c( \{e_2,e_4\}  )= 3 \\ 
c( \{e_1,e_2\} ) &= c( \{e_3,e_5\} ) = 4 \\ 
c( \{e_1,e_4\} ) &= c( \{e_2,e_3\} ) = 5.   
\end{align*}
Then $K_5$ together with the coloring $c$ is a uniformly colored
graph of type $(5,5,2)$ with degree $4.$
\end{example}

We will need the next lemma for the proof of Theorem
\ref{classification thm}.
\begin{lemma}\label{all isomorphic K5}
  Let $\frakn_1$ and $\frakn_2$ be two associate 
uniform Lie algebras defined by the uniformly colored graph
  in Example \ref{K5 near-one-factorization}. Then
$\frakn_1$ is isomorphic to $\frakn_2.$
\end{lemma}

\begin{proof}
 One choice of orientation of the graph gives the Lie algebra $\frakn_1$
with
\[ [v_3,v_4] = [v_2, v_5] = z_1, [v_4,v_5] = [v_3, v_1] = z_2,
[v_5,v_1] = [v_4, v_2] = z_3, \]
\[[v_1,v_2] = [v_5, v_3] = z_4, [v_2,v_3] = [v_1, v_4] = z_5. \]
Theorem B of \cite{payne-agag} can be used to show that all uniform
Lie algebras with this set of nonzero structure constants are
isomorphic to $\frakn_1$ or the Lie algebra $\frakn_2$ with uniform
basis $\{x_i\}_{i=1}^5 \cup \{w_j\}_{j=1}^5$ and Lie bracket
\[ [x_3,x_4] = [x_2, x_5] = w_1, [x_4,x_5] = [x_3, x_1] = w_2,
[x_5,x_1] = [x_4, x_2] = w_3, \]
\[[x_1,x_2] = [x_5, x_3] = w_4, [x_2,x_3] = -[x_1, x_4] = w_5 .\] 
 (See Remark \ref{Y hat}.)   But
these two Lie algebras are isomorphic through the isomorphism 
$\phi: \frakn_2 \to \frakn_1$ defined by
\[ \phi(x_1) =  -v_2, \phi(x_2) = -v_5, \phi(x_3) = -v_3, \phi(x_4) =
-v_1, \phi(x_5) = v_4, \text{and}\] 
\[ \phi(w_1) = z_2, \phi(w_2)  =- z_5, \phi(w_3) = -z_3,
\phi(w_4) = z_1, \phi(w_5) = z_4.\] 
Hence all sign choices yield isomorphic Lie algebras.
\end{proof}

There is an extensive body of research on graph factorizations,
one-factorizations and near-one-factorizations. 
\begin{remark}\label{lots}
Graphs may admit inequivalent one-factorizations.  For $n \le 3,$
$K_{2n}$ has a single one-factorization up to isomorphism.  For all $n
\ge 4,$ $K_{2n}$ has non-isomorphic one-factorizations. 
 The number  $F(2n)$  of nonisomorphic
one-factorizations
of $K_{2n} $ grows like $\ln F(2n) \sim 2n^2 \ln (2n)$
(\cite{cameron-75}),
showing that  there is a profusion of uniformly
colored graphs in higher dimensions.  
\end{remark}
 See \cite{lovasz-plummer, wallis-97,
plummer-07, wallis-07} for results on classes of
graphs admitting one-factorizations, and counts of nonisomorphic
factorizations of $K_{2n}$ for small $n.$

\section{Classification of uniform Lie algebras of 
type $(p,q,r)$ with $q \le 5$}\label{classification}

In Theorem \ref{classification thm} of this section we classify up to
isomorphism all uniform Lie algebras of type $(p,q,r)$ with $q \le 5.$
We begin by establishing some general classification results which will be
useful in the proof of Theorem \ref{classification thm}.  First we
show that any uniform Lie algebra of type $(p,q,r)$ with $p=1$ must be
Heisenberg.
\begin{prop}\label{p = 1}
Let $\frakn$ be a uniform Lie algebra of type $(1, q, r).$
Then $q=2r$ and $\frakn$ is isomorphic to the
$(2r+1)$-dimensional Heisenberg algebra.
\end{prop}
 The proposition may be proved by using the definition of
uniform Lie algebra, or by observing that any two-step
nilpotent Lie algebra with a one-dimensional center must be
the direct sum of a Heisenberg algebra and an abelian factor.
However, by Part \eqref{center} of Proposition \ref{basic
properties}, uniform Lie algebras do not have abelian
factors.  

Nikolayevsky has classified the two-step Einstein nilradicals with
two-dimensional center (\cite{nikolayevsky-2step}).  See also
\cite{lauret-oscari} in which nonsingular two-step Einstein nilradicals
with two-dimensional center are classified.  It is not hard to describe 
uniform   Lie algebra of type $(2, q, r)$ arising from connected
graphs.  

\begin{prop}\label{p = 2}
Let $\frakn$ be a uniform Lie algebra of type $(2, q, r).$
Suppose that the corresponding uniformly colored graph is
connected.  Then $q=2r,$ where $r \ge 2,$ $s=2,$ and
$\frakn$ and isomorphic to one of the Lie algebras $\frakn(2r+2)$
and  $\frakn'(2r+2)$ defined in Example \ref{Damek Ricci gen}.  
\end{prop}

\begin{proof}  
Let $(V,E)$ be the uniformly colored graph as in the statement of the
proposition.  The universal cover of this graph is the infinite line graph
$T_2$ uniformly colored with two colors.  The only finite uniformly
colored quotients are even cycle graphs with edges colored alternately
in two colors.

Theorem B of \cite{payne-agag} (see Remark \ref{Y hat}) can be used to
show that all choices of signs for the structure constants yield
either $\frakn(2r+2)$  or  $\frakn'(2r+2).$
\end{proof}

Now we reach our main classification theorem.
\begin{theorem}\label{classification thm}   
Suppose that $\frakn$ is a uniform Lie algebra of type $(p,q,r).$
If $q \le 5,$ then up to Lie algebra isomorphism, 
$\frakn$ occurs exactly once in Table \ref{list
of algebras}.
\end{theorem}
\begin{table}
\caption{Uniform Lie algebras of type $(p,q,r)$ with $q \le
  5$}
\label{list of algebras}
\setlength\extrarowheight{2pt}
\begin{center}
\begin{tabular}{|c|c|c|c|}
\hline
Case & $(p,q,r) $ & Defining bracket relations &  Description    \\
\hline
\hline
1 & $(1,2,1)$  &   $[v_1,v_2] = z_1$  &  $\frakh_{3}$ \\
\hline
2 & $(1,4,2)$ &$[v_1,v_2] =     [v_3,v_4] =  z_1$ &   $\frakh_5$ \\
\hline
2  & $(2,4,1)$ &  $[v_1,v_2] = z_1,[v_3, v_4] = z_2$ & $\frakh_3
                                                         \oplus
                                                         \frakh_3 \cong$ \\
4  & $(2,4,2)$  &  $[v_1,v_3] = [v_2,v_4] = z_1, [v_1,v_4] =  [v_2,
                  v_3] = z_2$ & Example \ref{r not invariant} \\
\hline
3 & $(3,3,1)$  & $[v_i,v_j] = z_{ij}, 1 \le i < j \le 3 $& $\frakf_{3,2}$ \\
\hline
4 &  $(4,4,1)$ & $[v_1,v_2] = z_1, [v_2,v_3] = z_2, [v_3,v_4] = z_3, [v_4,v_1] = z_4$ & Example \ref{cyclic algebra} \\
\hline
4 & $(2,4,2)$ &  $[v_1,v_2] = [v_3,v_4] = z_1, [v_2,v_3] =  [v_4, v_1] = z_2$ & Example \ref{Damek Ricci} \\
\hline
\multirow{2}{2mm}{5}  & \multirow{2}{13mm}{$(5,5,1)$} & $[v_1,v_2] = z_1,  [v_2,v_3] = z_2, [v_3,v_4] = z_3,$ & 
\multirow{2}{22mm}{Example \ref{cyclic algebra}}  \\
& & $[v_4,v_5]= z_4,
[v_5,v_1] = z_5$ &\\
\hline
\multirow{2}{2mm}{6} & \multirow{2}{13mm}{$(3,4,2)$} & $[v_1,v_2] =
                                                       [v_3,v_4] =
                                                       z_1,$ & \multirow{2}{22mm}{Example
  \ref{quaternionic}}\\
&  & $[v_1,v_3] = -[v_2,v_4]= z_2, [v_1,v_4] = [v_2, v_3] =     z_3$  & \\
\hline
\multirow{2}{2mm}{6} & \multirow{2}{16mm}{$(3,4,2)$} & $[v_1,v_2] = [v_3,v_4] = z_1,$ &  \multirow{2}{22mm}{Example
  \ref{quaternion associate}}  \\
& & $[v_1,v_3] = [v_2,v_4]= z_2, [v_1,v_4] = -[v_2, v_3] =   z_3$  &
                                                                     \\
\hline
6  & $(6,4,1)$  & $[v_i,v_j] = z_{ij}, 1 \le i < j \le 4 $ &  $\frakf_{4,2}$ \\
\hline
\multirow{3}{2mm}{7}&  & 
$[v_1,v_5] = [v_2,v_4] =  z_1,  [v_2,v_5] = [v_3,v_4] = z_2$ &   
  \\
&  $(5,5,2)$ & $ [v_1,v_3]
= [v_4,v_5] = z_3, [v_3,v_5]= [v_1,v_2] = z_4,$ & Example \ref{K5
                                                  near-one-factorization}
  \\
& &  $[v_1,v_4] =  [v_2, v_3]
= z_5$  & \\
\hline
 7 &$(10,5,1)$  &
$[v_i,v_j] = z_{ij}, 1 \le i < j \le 5 $ & $\frakf_{5,2}$  \\
\hline
\end{tabular}
\end{center}
\end{table}
\begin{proof}
Suppose that $\frakn$ is a uniform Lie algebra of type
$(p,q,r)$ with $q \le 5.$ Let $G = (V,E)$ be the associated
uniformly colored regular graph as in Definition
\ref{algebra to graph}.  By Proposition \ref{graph
properties}, $(V,E)$ is a regular graph with $q$ vertices
and $rp$ edges.  Such graph are classified; a list of all
such graphs is in Table 2.

\begin{table}\label{regular graphs}
\caption{Regular graphs with $q$ vertices and degree $s,$
for $q \le 5$ and $s \ge 1.$ The rightmost columns gives the
number of inequivalent uniform edge colorings and
the number $p$ of colors for each of those colorings.}
\begin{tabular}{| c|| c | c |c|c|c|}
\hline 
 Case & Degree $s$ & $q$ & Graph & Uniform colorings & $p$  \\
\hline
\hline
1 &1 &  2 &  $K_2$ & 1 & 1 \\
\hline
2 & 1 & 4 & $K_2 + K_2$ & 2 & 1,2 \\
\hline
3 & 2 & 3 &  $C_3$ & 1 & 3 \\
\hline
4 & 2 & 4 & $C_4$ & 2 & 2,4 \\
\hline
5 &2 & 5 & $C_5$ & 1  & 5 \\
\hline
6 & 3 & 4 &  $K_4$ &  2 & 3,6 \\
\hline
7 & 4 & 5 & $K_5$  & 2 & 5, 10 \\
\hline 
\end{tabular}
\end{table}

None of the regular graphs in the table have the same values of $p$
and $q,$ unless $(p,q) = (2,4).$ By Corollary \ref{invariants}, $p$
and $q$ are algebraic invariants, so we only need to show that any two
algebras in Table \ref{list of algebras} arising from the same graph
in Table 2 and having the same value of $p$ are
nonisomorphic, and that any two graphs with $(p,q) = (2,4)$ are
nonisomorphic.

We do a case by case analysis of the graphs in Table 2.
 We know from Proposition
\ref{constraints} that $2rp=sq, s \le p,$ and $p$ divides $|E|.$ The
first thing we do in each case is to use these three
constraints to determine a set of feasible values for $p.$ We then 
find  possible uniform edge colorings and the Lie algebras
associated to those colorings.  

\begin{case}  If $G = K_2,$ then $s=1$ and $q=2.$ The only
possible value for $p$ is one. By Proposition \ref{p = 1},
the corresponding uniform Lie algebra is three-dimensional
Heisenberg algebra $\frakh_3.$
\end{case}

\begin{case}  Suppose that $G = K_2 + K_2,$ so $s=1$ and $q=4.$
Then $p = 1$ or $p=2.$ If $p=1,$ then by Proposition \ref{p = 1}, the
corresponding uniform Lie algebra is five-dimensional Heisenberg
algebra $\frakh_5.$ If $p=2,$ then $r=1$ and we have the algebra
with $[v_1,v_2] = \pm z_1$ and $[v_3,v_4] = \pm z_2.$ By
Proposition \ref{r = 1}, all such algebras are isomorphic to
$\frakh_3 \oplus \frakh_3.$ 
\end{case}

\begin{case}
 If $G = C_3,$ then $s=2$ and $q=3,$ then $p=3$ and $r=1.$ The
underlying graph is complete and $p = (\begin{smallmatrix} q \\
2 \end{smallmatrix}).$ By Proposition \ref{r = 1}, all orientations
of the graph define isomorphic Lie algebras, and $\frakn$ is
isomorphic to $\frakf_{3,2}.$
\end{case}

\begin{case}  When $G = C_4,$ $s=2$ and $q=4.$ It follows that
$p=2$ or $p=4.$ If $p=2,$ Proposition \ref{p = 2} tells us that we have
either the 
six-dimensional Lie algebra in Example \ref{Damek Ricci}, or 
$\frakh_3 \oplus \frakh_3$ as in Example \ref{r not invariant}.  
These are not isomorphic, because the  Lie algebra in Example
\ref{Damek Ricci} is irreducible.    If
$p=4,$ then $r = 1,$ and by Proposition \ref{r = 1}, $\frakn$ is
isomorphic to the eight-dimensional cyclic Lie algebra $\frakm(4)$
as in Examples \ref{cyclic algebra} and \ref{cycle graph}.
\end{case}

\begin{case}  If $G = C_5,$ then $s=2$ and $q=5,$ then $p=5.$ By
Proposition \ref{r = 1}, all orientations of the graph yield
isomorphic Lie algebras.  Hence $\frakn$ is isomorphic to the
cyclic Lie algebra $\frakm(5)$ defined in Example \ref{cyclic
algebra}.
\end{case}

\begin{case}  If $G = K_4,$ then $s=3$ and $q=4.$ Either $p =3$ or
$p=6.$ If $p=3,$ then each color occurs $r=2$ times, and the coloring
defines a one-factorization of $K_4.$ Up to equivalence, there is a
unique one-factorization of $K_4.$ It yields uniform Lie algebras with
\begin{gather}\label{4,3}
 [v_1, v_2] = \pm z_1,  [v_3, v_4] =  \pm z_1, 
 [v_1, v_3]= \pm z_2, [v_2, v_4] = \pm z_2, \\
\notag [v_1, v_4]= \pm z_3,  [v_2, v_3] = \pm z_3.\end{gather}
as in the nonisomorphic Lie algebras from
 Examples \ref{quaternionic} and \ref{quaternion associate}.

We claim that any Lie algebra defined as above is isomorphic to either
the Lie algebra in Example \ref{quaternionic} or the Lie algebra in
Example \ref{quaternion associate}.  We use the method described in
Remark \ref{Y hat} to reduce the problem to considering four different
sign choices which are encoded in the sextuples
\begin{alignat*}{2}
 \bfs_1 &= (+,+,+,+,+,+), \qquad  
 \bfs_2 &= (+,+,+,+,+,-),  \\
 \bfs_3 &= (+,+,+,-,+,+),   \qquad
 \bfs_4 &= (+,+,+,-,+,-), 
\end{alignat*}
where the order of signs matched with the order of the brackets in
Equation \eqref{4,3}.  The signs in $\bfs_3$ define the Heisenberg
type Lie algebra in Example \ref{quaternionic}, while the sign in
$\bfs_2$ give the Lie algebra in Example \ref{quaternion
associate}.

The Lie algebras with signs as in $\bfs_1$ and $\bfs_4$ are
both isomorphic to the Lie algebra in Example \ref{quaternion
associate}, which has signs as in $\bfs_2.$ The change of basis 
\[ x_1 = v_1,  x_2 = v_3, x_3 = v_2, x_4 = v_4, w_1 = z_2, w_2 = z_1,
w_3 = z_3\]
converts from the basis $\{ v_i\}_{i=1}^4 \cup \{z_j\}_{j=1}^3$ with
signs as in $\bfs_1$ to a new basis $\{x_i\}_{i=1}^4 \cup
\{w_j\}_{j=1}^3$ with signs as in $\bfs_2.$ The change of basis 
\[ x_1 = v_2,  x_2 =- v_1, x_3 = v_3, x_4 = v_4, w_1 = z_1, w_2 = z_3,
w_3 = z_2\]
takes the basis $\{ v_i\}_{i=1}^4 \cup \{z_j\}_{j=1}^3$ with signs
as in $\bfs_1$ to the basis  $\{x_i\}_{i=1}^4 \cup
\{w_j\}_{j=1}^3$ with signs as in $\bfs_4.$ Thus, if $p=3,$ then
$\frakn$ is isomorphic to either the Lie algebra from 
Example \ref{quaternionic} or the Lie algebra in
 \ref{quaternion associate}.

If $p=6,$ then each edge of the graph is a different color.  By
Proposition \ref{r = 1}, $\frakn$ is isomorphic to $\frakf_{4,2}.$
\end{case}

\begin{case}  If $G = K_5,$ then $s=4$ and $q=5,$ so that
$p = 5$ or $p = 10.$ If $p = 5,$ then the coloring is a
near-one-factorization of $K_5.$ Up to equivalence $K_5$ has
a unique near-one-factorization.  Hence the corresponding
undirected colored graph is the same as the one described
in Example \ref{K5 near-one-factorization}.  By Lemma
\ref{all isomorphic K5}, all possible sign choices yield
isomorphic Lie algebras. If $p = 10,$ then $p =
(\begin{smallmatrix} q \\ 2 \end{smallmatrix}),$ so each
edge of the graph is a different color.  By Proposition \ref{r
= 1}, $\frakn$ is isomorphic to $\frakf_{5,2}.$
\end{case}  \vskip -.2in
 \end{proof}

\bibliographystyle{acm}



\end{document}